\documentclass[11pt]{article}

\usepackage[margin=1in]{geometry}
\usepackage{amsmath,amssymb,amsthm,mathtools,bm}
\usepackage{hyperref}
\usepackage{algorithm}
\usepackage{algpseudocode}
\usepackage[capitalize,nameinlink]{cleveref}

\usepackage{todonotes}

\newtheorem{theorem}{Theorem}[section]
\newtheorem{proposition}[theorem]{Proposition}
\newtheorem{remark}[theorem]{Remark}
\newtheorem{definition}[theorem]{Definition}
\newtheorem{example}[theorem]{Example}

\newcommand{\R}{\mathbb R}

\DeclareMathOperator{\supp}{supp}

\providecommand{\keywords}[1]
{
  \small	
  {\textit{Keywords---}} #1
}

\begin{document}

\title{\bf A Discrete Cumulative Distribution Transform via Optimal Transport}
\author{{\sc Harbir Antil$^1$, Gustavo Rohde$^2$ and Aryan Saxena$^1$} \\  \\ 
$^1$Center for Mathematics and Artificial Intelligence and \\
Department of Mathematical Sciences, 
George Mason University, Fairfax VA 22030. \\
$^2$Department of Biomedical Engineering and Electrical and \\ 
Computer Engineering University of Virginia 
Charlottesville, VA 22908, USA. 
\thanks{HA and AS are partially supported by the Office of Naval Research (ONR) under Award NO: N00014-24-1-2147. NSF grant DMS-2408877, and the Air Force Office of Scientific Research (AFOSR) under Award NO: FA9550-25-1-0231. GR is partially supported by NIH award GM130825, and ONR award N000142212505.}}
\date{}
\maketitle

\begin{abstract}
This paper develops a fully discrete cumulative distribution transform
(CDT) for atomic probability measures on the real line. The transform is
defined through monotone quantile maps and admits explicit linear-time
algorithms for both forward transformation and inverse reconstruction
based solely on cumulative mass matching. Unlike the classical continuous
setting, deterministic transport between atomic measures cannot generally
split masses, so exact reconstruction may fail at finite resolution. We
establish a precise cumulative-mass compatibility criterion for exact finite-resolution recovery and prove weak convergence of reconstructed measures under reference refinement. Several structural properties of the discrete CDT
are derived, including translation, composition, and scaling laws, and
the framework is extended to a discrete signed cumulative distribution
transform with thresholded stabilization near zero crossings. By avoiding
continuous interpolation, the proposed framework provides a simple
fixed-reference transport representation for discrete data. Numerical
examples illustrate translation linearization, compatibility-controlled
reconstruction, refinement consistency, and stabilization of the signed
transform.
\end{abstract}


\section{Introduction}

The cumulative distribution transform (CDT) is a transport-based signal
representation derived from monotone optimal transport on the real line.
Given a reference probability measure and a target measure, the CDT
represents the target through the unique monotone transport map pushing
the reference onto the target. Since its introduction, the CDT and its
extensions have attracted considerable attention in signal processing,
pattern recognition, imaging, inverse problems, and machine learning
because several nonlinear geometric transformations become linear in
transport coordinates \cite{kolouri2018cdt,kolouri2017optimal,aldroubi2021signed}. In particular, translations,
scalings, and certain monotone deformations are transformed into simple
algebraic operations on the associated transport maps. These properties have enabled new, computational and data efficient, solutions a variety of problems and applications including parametric signal estimation \cite{rubaiyat2020parametric,nichols2019time}, communications \cite{park2018multiplexing,neary2021transport}, system identification \cite{rubaiyat2024data}, structural health monitoring \cite{wang2020fault,zhang2024combining,rubaiyat2020parametric, morgan2025corrosion}, reduced order modeling solutions of PDEs \cite{ren2021model}, signal classification \cite{rubaiyat2024end} and others.  

The original CDT \cite{kolouri2018cdt} was formulated primarily for absolutely continuous
probability densities using continuous cumulative distribution functions
and generalized inverse maps. This continuous framework is closely tied
to one-dimensional optimal transport and monotone rearrangement
theory \cite{villani2021topics,FSantambrogio_2015a,ollivier2014optimal,
ambrosio2005gradient}. However, many datasets arising in practice are
inherently discrete, including empirical measures, sampled signals,
histograms, point clouds, particle approximations, and discretized
probability distributions. In such settings, interpolation and density
estimation may introduce artificial smoothing, numerical ambiguity, and
additional computational complexity.

The purpose of this paper is to develop a mathematically rigorous and fully 
discrete cumulative distribution transform (CDT) directly for atomic probability 
measures on the real line. Given a discrete reference measure 
$\sigma = \sum_{j=1}^m q_j \delta_{y_j}$ with sorted support 
$y_1 < y_2 < \dots < y_m$, and a discrete target measure 
$\mu = \sum_{i=1}^n p_i \delta_{x_i}$ with sorted support 
$x_1 < x_2 < \dots < x_n$, we define the discrete CDT through the monotone 
quantile representative
\[
T_\mu = F_\mu^{-1} \circ F_\sigma.
\]
This formulation yields explicit cumulative-mass characterizations and simple 
linear-time algorithms requiring no interpolation, density estimation, or 
continuous inverse-CDF approximation. Crucially, by fixing a reference measure 
with $m$ atoms, the transform maps any arbitrary target measure to a 
fixed-dimensional vector $T_\mu \in \R^m$ regardless of $n$. This provides a fixed-reference transport representation of uniform
dimension that is naturally suited for comparison and downstream
learning tasks.

A central theme of the paper is the distinction between the discrete quantile 
representation and the reconstructed pushforward measure induced by the 
transport map. Unlike the classical continuous setting, deterministic transport 
maps between atomic measures cannot generally split masses---a well-known 
bottleneck that typically forces a pivot from Monge transport maps to 
Kantorovich transport plans in standard optimal transport 
literature~\cite{FSantambrogio_2015a}. In this work, we deliberately retain a 
strict vector-valued map formulation and treat the resulting discrepancy as an 
intrinsic feature of the representation. Consequently, the reconstructed 
measure $(T_\mu)_\#\sigma$ does not necessarily coincide with the original 
target measure~$\mu$. We characterize a precise cumulative-mass compatibility criterion under
which exact finite-resolution reconstruction is possible and show that
failure of this condition is an inherent obstruction of deterministic
atomic transport. We further prove that this obstruction disappears
asymptotically under reference refinement: as the maximal reference mass
tends to zero, the reconstructed measures converge weakly to the target.

In addition to the forward transform, we develop explicit inverse 
reconstruction algorithms and establish several structural properties of the 
discrete CDT, including translation, composition, and scaling laws. We also 
extend the framework to a discrete signed cumulative distribution transform 
(SCDT), introducing a thresholded decomposition designed to regularize the 
positive-negative Jordan splitting. This thresholding stabilizes the transport 
representation against high-frequency oscillations or small numerical 
perturbations near zero crossings, providing a robust discrete counterpart to 
the continuous signed transport framework introduced in~\cite{aldroubi2021signed}.

The resulting framework provides a simple and efficient quantile-based
transport representation for atomic probability measures that is
naturally suited for empirical and sampled data. While related
generalized-inverse computations already arise in existing numerical
implementations of CDT and SCDT, the present work provides an explicit
discrete monotone transport interpretation together with a rigorous
analysis of reconstruction, compatibility, refinement consistency, and
signed zero-crossing effects.

The remainder of the paper is organized as follows.
Section~\ref{sec:monotone} reviews monotone transport on the real line.
Section~\ref{sec:dcdt} introduces the discrete CDT and derives explicit
forward algorithm. Inverse discrete CDT is studied in Section~\ref{s:iDCDT}. 
Section~\ref{sec:properties} establishes structural properties and proves
reconstruction consistency under refinement.
Section~\ref{sec:scdt} develops the discrete signed CDT.
Finally, Section~\ref{sec:numerics} presents numerical illustrations
demonstrating translation linearization, compatibility-controlled
reconstruction, refinement consistency, and stabilization near zero
crossings.

\section{Preliminaries: Monotone Transport}
\label{sec:monotone}

Let \(\mathcal P(\R)\) denote the set of Borel probability measures on \(\R\).
For \(\mu\in\mathcal P(\R)\), its cumulative distribution function (CDF) is defined by
\[
F_\mu(x)=\mu((-\infty,x]), \qquad x\in\R.
\]
The function \(F_\mu\) is nondecreasing, right-continuous, and satisfies
\[
\lim_{x\to-\infty}F_\mu(x)=0, \qquad \lim_{x\to+\infty}F_\mu(x)=1.
\]
The generalized inverse (or quantile function) of \(F_\mu\) is defined for \(s\in(0,1]\) by
\begin{equation}\label{eq:genearlized_inverse}
F_\mu^{-1}(s) = \inf\{x\in\R:\ F_\mu(x)\ge s\}.
\end{equation}
The map \(F_\mu^{-1}\) is nondecreasing and left-continuous.

If \(T:\R\to\R\) is measurable and \(\sigma\in\mathcal P(\R)\), the pushforward of \(\sigma\) under \(T\) is the probability measure \(T_\#\sigma\in\mathcal P(\R)\) defined by
\begin{equation}\label{eq:pushforward}
(T_\#\sigma)(A) = \sigma(T^{-1}(A))
\end{equation}
for every Borel set \(A\subset\R\). Equivalently, \(T_\#\sigma\) is the distribution obtained by transporting the mass of \(\sigma\) through the map \(T\).

Next we state the classical monotone rearrangement theorem whose proof can be found in \cite[Section~2.1]{FSantambrogio_2015a}. 

\begin{theorem}[Monotone transport on the line]
Let \(\sigma,\mu\in\mathcal P(\R)\), and suppose that \(\sigma\) is nonatomic. Then there exists a nondecreasing measurable map \(T:\R\to\R\) satisfying
\[
T_\#\sigma=\mu.
\]
Moreover, \(T\) is unique \(\sigma\)-almost everywhere and is given by
\[
T = F_\mu^{-1}\circ F_\sigma.
\]
\end{theorem}

\section{Discrete CDT}
\label{sec:dcdt}

\subsection{Discrete reference measures}

We now develop a fully discrete analogue of the cumulative distribution transform for atomic probability measures on the real line. Let
\[
\sigma = \sum_{j=1}^m q_j\delta_{y_j}, \qquad y_1 < \dots < y_m,
\quad 
\mbox{where}
\quad 
q_j > 0, \qquad \sum_{j=1}^m q_j = 1.
\]
The measure \(\sigma\) will be called the \emph{reference measure}. The support locations \(y_j\) define the reference grid, while the weights \(q_j\) define the reference mass distribution.
Similarly, let
\[
\mu = \sum_{i=1}^n p_i\delta_{x_i}, \qquad x_1 < \dots < x_n,
\quad 
\mbox{with}
\quad 
p_i > 0, \qquad \sum_{i=1}^n p_i = 1.
\]

\begin{remark}[Common reference measure]
\rm 
The reference measure \(\sigma\) is fixed throughout the transform construction. All discrete CDT representations are therefore computed relative to the same reference cumulative mass grid. This common reference structure is essential: it places all transformed signals in a canonical coordinate system and makes pointwise comparison of different transforms meaningful.
\end{remark}

\subsection{Definition of the discrete CDT}

\begin{definition}[Discrete CDT]
\label{def:DCDT}
The discrete cumulative distribution transform of \(\mu\) relative to the reference measure \(\sigma\) is the monotone quantile map
\[
\mathcal C_\sigma(\mu) = T_\mu,
\]
defined by
\[
T_\mu = F_\mu^{-1}\circ F_\sigma .
\]
\end{definition}
Since \(F_\sigma\) is constant between consecutive reference atoms, the transform is completely determined by its values on the reference support \(\{y_j\}_{j=1}^m\).

\subsection{Explicit computable formula}

Define cumulative sums
\[
Q_j = \sum_{\ell=1}^j q_\ell, \qquad P_i = \sum_{k=1}^i p_k.
\]

\begin{proposition}[Discrete quantile formula]
\label{prop:quantileFormula}
For every \(j\in\{1,\dots,m\}\),
\[
T_\mu(y_j) = x_i, \qquad i=\min\{k:\ P_k\ge Q_j\}.
\]
\end{proposition}

\begin{proof}
By definition of the discrete CDT,
\[
T_\mu(y_j) = F_\mu^{-1}(F_\sigma(y_j)).
\]
Since
\[
F_\sigma(y_j) = \sum_{\ell=1}^j q_\ell = Q_j,
\]
we obtain
\[
T_\mu(y_j) = F_\mu^{-1}(Q_j).
\]
By the definition of the generalized inverse (cf.~\eqref{eq:genearlized_inverse}),
$
F_\mu^{-1}(s) = \inf\{x\in\R:\ F_\mu(x)\ge s\}.
$
Since \(\mu=\sum_{i=1}^n p_i\delta_{x_i}\), the cumulative distribution function \(F_\mu\) is the step function
\[
F_\mu(x) =
\begin{cases}
0, & x<x_1,\\
P_i, & x_i\le x<x_{i+1},\quad i=1,\dots,n-1,\\
P_n=1, & x\ge x_n.
\end{cases}
\]
Let
\[
i=\min\{k:\ P_k\ge s\}.
\]
By minimality of \(i\), we have \(P_{i-1}<s\le P_i\), where we adopt the convention \(P_0=0\). Since \(F_\mu\) is a step function, we have
\[
F_\mu(x) \le P_{i-1} < s \qquad \text{for all } x<x_i.
\]
On the other hand,
\[
F_\mu(x_i)=P_i\ge s.
\]
Therefore, \(x_i\) is the smallest point at which \(F_\mu(x)\ge s\). Hence, \(F_\mu^{-1}(s)=x_i\). Applying this with \(s=Q_j\) gives
\[
T_\mu(y_j)=F_\mu^{-1}(Q_j) = x_i, \qquad i=\min\{k:\ P_k\ge Q_j\}.
\]
The proof is complete.
\end{proof}

Thus, the discrete CDT is obtained by matching each reference cumulative mass level \(Q_j\) with the first target support location whose cumulative mass reaches or exceeds that level.

\subsection{Implementable Discrete CDT Algorithm}

The discrete CDT can be computed efficiently by a monotone cumulative-mass matching procedure. The key observation is that the quantile formula in Proposition~\ref{prop:quantileFormula} requires only cumulative sums and a monotone sweep through the ordered support points.
Recall that
$
Q_j = \sum_{\ell=1}^j q_\ell,$ $P_i = \sum_{k=1}^i p_k.
$
The algorithm proceeds by advancing through the target cumulative masses until the target cumulative level \(P_i\) reaches the reference cumulative level \(Q_j\).

\begin{algorithm}[H]
\caption{Discrete cumulative distribution transform}
\label{alg:dcdt}
\begin{algorithmic}[1]

\Require Reference measure $\sigma=\sum_{j=1}^m q_j\delta_{y_j}$, $y_1<\dots<y_m$; target measure $\mu=\sum_{i=1}^n p_i\delta_{x_i}$, $x_1<\dots<x_n$.
\Ensure Discrete CDT values $T_j=T_\mu(y_j)$, $j=1,\dots,m$.

\State Initialize $i \gets 1$, $P \gets p_1$, $Q \gets 0$.
\For{$j=1,\dots,m$}
    \State Update the reference cumulative mass: $Q \gets Q + q_j$.
    \While{$P < Q$ and $i < n$}
        \State Advance to the next target cumulative level: $i \gets i + 1$, $P \gets P + p_i$.
    \EndWhile
    \If{$P < Q$ and $i = n$} \Comment{Floating-point safeguard at the final mass level}
        \State $P \gets Q$.
    \EndIf
    \State Set $T_j \gets x_i$.
\EndFor
\State \Return $\{T_j\}_{j=1}^m$

\end{algorithmic}
\end{algorithm}

Because both cumulative mass indices advance monotonically and never backtrack, each support point is visited at most once. Therefore, the total computational complexity is \(\mathcal O(m+n)\).

This non-backtracking, two-pointer cumulative mass matching procedure is 
algorithmically reminiscent of classical methods used to compute the 
one-dimensional Wasserstein distance or optimal transport couplings between 
discrete measures (see, e.g., \cite[Section 2.6]{GPeyre_MCuturi_2019a}). 
However, a crucial distinction remains: while standard computational optimal 
transport frameworks focus on generating an optimal transport joint plan 
(a matrix allowing mass-splitting) to evaluate a distance metric, Algorithm~\ref{alg:dcdt} 
strictly constructs a deterministic, vector-valued transport map 
$T_\mu \in \R^m$ to serve as a fixed-dimensional linear representation.

\begin{remark}[Floating-point safeguard]
\rm 
In exact arithmetic, the safeguard is never activated since the total reference and target masses both equal one. Its sole purpose is numerical robustness at the final cumulative mass level, where floating-point roundoff may produce a negligible discrepancy between the accumulated values of \(P\) and \(Q\).
\end{remark}

\section{Inverse Discrete CDT}
\label{s:iDCDT}

The discrete CDT in Algorithm~\ref{alg:dcdt} produces the monotone quantile representative
\[
T_\mu
=
F_\mu^{-1}\circ F_\sigma
\]
defined on the support of the reference measure
$
\sigma=\sum_{j=1}^m q_j\delta_{y_j}.
$
A natural question is whether the target measure \(\mu\) can be recovered
from the transport map \(T_\mu\) by pushforward through the reference
measure.
To investigate this, define the reconstructed measure
\[
\widetilde\mu
:=
(T_\mu)_\#\sigma.
\]

\begin{proposition}[Inverse reconstruction formula]
\label{prop:invDCDT}
Let
$
T_\mu(y_j)=F_\mu^{-1}(Q_j).
$
Then the reconstructed measure satisfies
\[
\widetilde\mu
=
(T_\mu)_\#\sigma
=
\sum_{j=1}^m q_j\delta_{T_\mu(y_j)}.
\]
\end{proposition}

\begin{proof}
By definition of pushforward measure (cf.~\eqref{eq:pushforward}),
$
((T_\mu)_\#\sigma)(A)
=
\sigma(T_\mu^{-1}(A))
$
for every Borel set \(A\subset\R\).
Since
$
\sigma
=
\sum_{j=1}^m q_j\delta_{y_j},
$
we obtain
\[
((T_\mu)_\#\sigma)(A)
=
\sum_{j=1}^m
q_j
\delta_{y_j}(T_\mu^{-1}(A)).
\]
Now observe that
\[
\delta_{y_j}(T_\mu^{-1}(A))
=
\begin{cases}
1, & T_\mu(y_j)\in A,\\
0, & T_\mu(y_j)\notin A,
\end{cases}
\]
which is precisely
\[
\delta_{T_\mu(y_j)}(A).
\]
Therefore,
\[
((T_\mu)_\#\sigma)(A)
=
\sum_{j=1}^m
q_j
\delta_{T_\mu(y_j)}(A).
\]
Since this identity holds for every Borel set \(A\subset\R\),
we conclude that
\[
(T_\mu)_\#\sigma
=
\sum_{j=1}^m q_j\delta_{T_\mu(y_j)}.
\]
The proof is complete. 
\end{proof}

Thus inverse reconstruction transports each reference atom from its
original location \(y_j\) to the transported location \(T_\mu(y_j)\)
while preserving its mass \(q_j\).
The reconstruction therefore redistributes reference mass without
creating or splitting atoms.
If several reference atoms are mapped to the same location, the
corresponding masses combine automatically.
For example, if
\[
T_\mu(y_{j_1})=T_\mu(y_{j_2})=x,
\]
then the reconstructed measure contains the atom
\[
(q_{j_1}+q_{j_2})\delta_x.
\]

\begin{theorem}[Exact reconstruction criterion]
\label{thm:exactReconstruction}
Let \(\sigma=\sum_{j=1}^m q_j\delta_{y_j}\) be the reference measure with strictly positive masses and cumulative levels \(Q_j=\sum_{\ell=1}^j q_\ell\), and let \(\mu=\sum_{i=1}^n p_i\delta_{x_i}\) be the target measure with strictly positive masses and cumulative levels \(P_i=\sum_{k=1}^i p_k\). Let \(\widetilde\mu=(T_\mu)_\#\sigma\) denote the reconstructed measure induced by the discrete CDT, and let 
\[
J_i = \{j\in\{1,\dots,m\}:\ P_{i-1}<Q_j\le P_i\}, \qquad P_0=0,
\]
denote the block of reference indices transported to each target location \(x_i\).

Then exact reconstruction holds, that is, \(\widetilde\mu=\mu\), if and only if every target cumulative mass level \(P_i\) is also a reference cumulative level:
\[
\{P_1,\dots,P_n\}\subseteq\{Q_1,\dots,Q_m\}.
\]
Equivalently, since \(P_n=Q_m=1\), it is enough to require \(\{P_1,\dots,P_{n-1}\}\subseteq\{Q_1,\dots,Q_m\}\). In this case, the target masses are exactly recovered by the aggregated reference block masses:
\[
p_i=\sum_{j\in J_i}q_j \qquad \text{for each } i=1,\dots,n.
\]
\end{theorem}

\begin{remark}[Connection to the Monge Problem]
\rm 
Theorem~\ref{thm:exactReconstruction} delineates the precise boundary where a deterministic transport map can push forward one discrete measure into another without splitting mass. In classical optimal transport, this subset condition dictates whether a discrete 1D Monge map exists (see \cite[Section 1.1]{FSantambrogio_2015a} or \cite[Section 2.2]{rockafellar2014superquantile}). When it is violated, exact matching can only be accomplished via a mass-splitting Kantorovich transport plan.
\end{remark}

\begin{proof}
By Proposition~\ref{prop:quantileFormula},
\[
T_\mu(y_j)=x_i \quad
\mbox{when}
\quad 
i=\min\{k:\ P_k\ge Q_j\}.
\]
Since the cumulative masses \(P_k\) are strictly increasing, this
minimality condition is equivalent to requiring that \(Q_j\) lies between
the consecutive cumulative levels \(P_{i-1}\) and \(P_i\), that is,
\[
P_{i-1}<Q_j\le P_i,
\]
with the convention \(P_0=0\). Thus the reference atoms transported to \(x_i\) are precisely those with indices in \(J_i=\{j:\ P_{i-1}<Q_j\le P_i\}\). 

By Proposition~\ref{prop:invDCDT}, we have 
\(
\widetilde\mu=(T_\mu)_\#\sigma,
\)
the reconstructed mass assigned to any Borel set \(A\subset\R\) is
\[
\widetilde\mu(A)=\sigma(T_\mu^{-1}(A)).
\]
Because
\(
\sigma=\sum_{j=1}^m q_j\delta_{y_j},
\)
this becomes
\[
\widetilde\mu(A)
=
\sum_{\{j:\,T_\mu(y_j)\in A\}} q_j.
\]
Taking \(A=\{x_i\}\), we obtain
\[
\widetilde\mu(\{x_i\})
=
\sum_{\{j:\,T_\mu(y_j)=x_i\}} q_j.
\]
Since
\(
T_\mu(y_j)=x_i
\iff j\in J_i,
\)
it follows that
\[
\widetilde\mu(\{x_i\})
=
\sum_{j\in J_i} q_j.
\]
Suppose first that
\(
\{P_1,\dots,P_n\}\subseteq\{Q_1,\dots,Q_m\}.
\)
Since \(P_0=0\), \(P_n=1=Q_m\), and each \(P_i\) coincides with some
reference cumulative level, there exist indices
$
0=j_0<j_1<\cdots<j_n=m
$
such that
$
P_i=Q_{j_i}
$
$\text{for each }i=0,\dots,n.
$
Substituting this into the definition
$
J_i=\{j:\ P_{i-1}<Q_j\le P_i\}
$
gives
\[
J_i=\{j:\ Q_{j_{i-1}}<Q_j\le Q_{j_i}\}.
\]
Since the cumulative levels \(Q_j\) are strictly increasing, this is
equivalent to
\[
j_{i-1}<j\le j_i,
\]
that is,
\[
J_i=\{j_{i-1}+1,\dots,j_i\}.
\]
Therefore,
\[
\widetilde\mu(\{x_i\})
=
\sum_{j=j_{i-1}+1}^{j_i}q_j
=
Q_{j_i}-Q_{j_{i-1}}
=
P_i-P_{i-1}
=
p_i.
\]
Thus \(\widetilde\mu\) assigns mass \(p_i\) to each target location
\(x_i\), and no mass elsewhere. Hence
\(
\widetilde\mu=\mu.
\)

Conversely, suppose \(\widetilde\mu=\mu\). Fix an arbitrary index \(k\in\{1,\dots,n\}\). Because the reconstructed masses agree with the true target masses at every atom, we can evaluate their cumulative sums simultaneously:
\[
P_k = \sum_{i=1}^k p_i = \sum_{i=1}^k \widetilde\mu(\{x_i\}) = \sum_{i=1}^k \sum_{j\in J_i}q_j.
\]
The index sets \(J_i=\{j:\ P_{i-1}<Q_j\le P_i\}\) are contiguous and mutually disjoint. Because the intervals \((P_{i-1},P_i]\) form a partition of \((0,P_k]\), their index union over \(i=1,\dots,k\) is
\[
\bigcup_{i=1}^k J_i = \{j\in\{1,\dots,m\}:\ 0<Q_j\le P_k\}.
\]
Since \(P_k>0\), the sum
$
P_k=\sum_{i=1}^k\sum_{j\in J_i}q_j
$
contains at least one strictly positive term. Hence
$
\{j:\ 0<Q_j\le P_k\}
$
is nonempty, so the index
$
b_k=\max\{j:\ Q_j\le P_k\}
$
is well-defined. 
Because the cumulative levels \(Q_j\) are strictly increasing, whenever
\(Q_j\le P_k\) holds for some index \(j\), it also holds for every
earlier index \(\ell\le j\). Since \(b_k\) is the largest index
satisfying \(Q_{b_k}\le P_k\), it follows that
\[
\{j\in\{1,\dots,m\}:\ Q_j\le P_k\}
=
\{1,\dots,b_k\}.
\]
Therefore,
\[
P_k
=
\sum_{j=1}^{b_k}q_j
=
Q_{b_k}.
\]
Since \(b_k\in\{1,\dots,m\}\), this shows that
$
P_k=Q_{b_k}\in\{Q_1,\dots,Q_m\}.
$
Since \(k\) was arbitrary, it follows that
$
P_k\in\{Q_1,\dots,Q_m\}
$
$\text{for every }k=1,\dots,n.
$
Hence
$
\{P_1,\dots,P_n\}\subseteq\{Q_1,\dots,Q_m\},
$
which completes the proof.
\end{proof}

\begin{remark}[Mass compatibility and reference design]
\rm 
Theorem~\ref{thm:exactReconstruction} shows that exact reconstruction
fails precisely when some target mass boundary falls strictly between two
reference cumulative levels, thereby requiring a reference atom to be
split into smaller pieces. Since deterministic discrete transport
preserves the mass of each reference atom, such splitting is impossible.

This characterization provides a practical criterion for reference
design: exact invertibility can be verified a priori by checking whether
the target cumulative mass levels are contained in the reference
cumulative levels. When this condition fails, refinement improves
reconstruction by reducing cumulative-mass quantization error.
\end{remark}

\begin{example}[Failure and recovery of exact reconstruction]
\rm 
Consider first the reference measure
$
\sigma=\delta_0,
$
and the target measure
$
\mu=\frac12\delta_{-1}+\frac12\delta_1.
$
Since the reference measure consists of a single atom carrying total mass
\(1\), we have
$
Q_1=1.
$
Therefore,
\[
T_\mu(0)
=
F_\mu^{-1}(1)
=
1.
\]
Consequently, the reconstructed measure is
\[
\widetilde\mu
=
(T_\mu)_\#\sigma
=
\delta_1,
\]
which does not equal \(\mu\).
The obstruction is that the single reference atom carries total mass
\(1\), and a deterministic transport map cannot split this mass between
the two target locations \(-1\) and \(1\).

However, the situation changes if the reference measure contains
sufficiently many atoms.
For example, let
$
\sigma
=
\frac12\delta_0+\frac12\delta_1,
$
while keeping the same target measure
$
\mu
=
\frac12\delta_{-1}
+
\frac12\delta_1.
$
Then
$
Q_1=\frac12,
$
$
Q_2=1.
$
Hence
$
T_\mu(0)
=
F_\mu^{-1}\!\left(\frac12\right)
=
-1,
$
and
$
T_\mu(1)
=
F_\mu^{-1}(1)
=
1.
$
Therefore,
$
(T_\mu)_\#\sigma
=
\frac12\delta_{-1}
+
\frac12\delta_1
=
\mu.
$
This illustrates that exact reconstruction is governed entirely by mass
compatibility between the reference and target atomic decompositions.
\end{example}

This phenomenon highlights a fundamental departure from the continuous 
setting. In the classical continuous CDT  framework \cite{kolouri2018cdt}, 
the transform is a guaranteed 
bijective mapping whenever the reference and target measures possess 
strictly positive probability density functions. For fully discrete measures, 
however, bijectivity breaks down because deterministic transport maps 
cannot split mass; exact invertibility is thus strictly conditional on mass 
decomposition compatibility as characterized by Theorem~\ref{thm:exactReconstruction}.

The previous results describe reconstruction at the measure-theoretic
level. We now formulate the corresponding linear-time computational
procedure.
Given the discrete CDT values
\[
T_j=T_\mu(y_j),
\qquad j=1,\dots,m,
\]
the inverse discrete CDT reconstructs the pushforward measure
\[
\widetilde\mu=(T_\mu)_\#\sigma.
\]
In general, \(\widetilde\mu\) need not coincide with the original target
measure \(\mu\).
The inverse reconstruction transports each reference mass \(q_j\) from the
location \(y_j\) to the transported location \(T_j\).
Although the reconstructed measure is explicitly given by
(cf.~Proposition~\ref{prop:invDCDT})
\[
\widetilde\mu
=
\sum_{j=1}^m q_j\delta_{T_j},
\]
distinct reference atoms may be transported to the same target location.
The following algorithm computes the aggregated atomic representation of
\(\widetilde\mu\).

\begin{algorithm}[H]
\caption{Inverse discrete cumulative distribution transform}
\begin{algorithmic}[1]

\Require
Reference masses and locations
$
\{(q_j,y_j)\}_{j=1}^m,
$
and discrete CDT values
$
\{T_j\}_{j=1}^m.
$

\Ensure
Aggregated reconstructed measure
$
\widetilde\mu
=
\sum_z \widetilde p_z\delta_z.
$

\State Initialize an empty list of target atoms and masses.

\For{\(j=1,\dots,m\)}

    \If{\(j=1\) or \(T_j\neq T_{j-1}\)}

        \State Create a new target atom at location \(T_j\) with mass \(q_j\).

    \Else

        \State Add mass \(q_j\) to the existing atom at location \(T_j\).

    \EndIf

\EndFor

\State Return
$
\widetilde\mu
=
\sum_z \widetilde p_z\delta_z,
$
where
$
\widetilde p_z
=
\sum_{\{j:\ T_j=z\}} q_j.
$

\end{algorithmic}
\end{algorithm}

Thus, if several indices \(j\) produce the same transported location
\(T_j=z\), then their masses are aggregated:
\[
\widetilde p_z
=
\sum_{\{j:\ T_j=z\}}q_j.
\]
Because the transport map \(T_\mu\) is nondecreasing, the sequence
$
\{T_j\}_{j=1}^m
$
is naturally ordered.
Consequently, identical transported locations occur contiguously, and the
aggregation can be performed by a single linear sweep through the
reference atoms.
Therefore, the total computational complexity of the inverse algorithm is
$
\mathcal O(m).
$
No additional sorting or search operations are required beyond the single
monotone sweep through the transported atoms.

\begin{remark}[Forward--inverse complexity]
\rm 
The forward and inverse discrete CDT algorithms have respective
complexities \(\mathcal O(m+n)\) and \(\mathcal O(m)\), making the
entire discrete transform pipeline computationally linear in the number
of atoms.
\end{remark}

\section{Discrete CDT properties}
\label{sec:properties}

\subsection{Translation property}

One of the fundamental properties of the CDT is the linearization of
translations.

\begin{proposition}[Translation property]
Let \(a\in\R\), and define the translated target measure
\[
\mu_a
=
(\tau_a)_\#\mu,
\qquad
\tau_a(x)=x+a.
\]
Then the corresponding discrete CDT satisfies
\[
T_{\mu_a}(y_j)
=
T_\mu(y_j)+a,
\qquad j=1,\dots,m.
\]
Equivalently,
\[
T_{\mu_a}
=
T_\mu+a,
\]
where the addition is understood pointwise on the reference support.
\end{proposition}

\begin{proof}
Since
$
\mu=\sum_{i=1}^n p_i\delta_{x_i},
$
the translated target measure is
$
\mu_a
=
\sum_{i=1}^n p_i\delta_{x_i+a}.
$
Its cumulative distribution function satisfies
\[
F_{\mu_a}(x)
=
\mu_a((-\infty,x])
=
\mu((-\infty,x-a])
=
F_\mu(x-a).
\]
For \(s\in(0,1]\), the generalized inverse is
\[
F_{\mu_a}^{-1}(s)
=
\inf\{x\in\R:\ F_{\mu_a}(x)\ge s\}.
\]
Using \(F_{\mu_a}(x)=F_\mu(x-a)\), we obtain
\[
F_{\mu_a}(x)\ge s
\iff
F_\mu(x-a)\ge s.
\]
Setting \(y=x-a\), we get
\[
F_{\mu_a}^{-1}(s)
=
\inf\{y+a:\ F_\mu(y)\ge s\}
=
a+\inf\{y:\ F_\mu(y)\ge s\}.
\]
Therefore,
\[
F_{\mu_a}^{-1}(s)
=
F_\mu^{-1}(s)+a.
\]
Evaluating this identity at the reference cumulative mass levels \(Q_j\)
gives
\[
T_{\mu_a}(y_j)
=
F_{\mu_a}^{-1}(Q_j)
=
F_\mu^{-1}(Q_j)+a
=
T_\mu(y_j)+a.
\]
The proof is complete. 
\end{proof}

\begin{remark}[Connection with translated discrete signals]
\rm 
If a discrete signal is represented by the atomic measure
$
\mu=\sum_{i=1}^n w_i\delta_{x_i},
$
then spatial translation of the signal corresponds precisely to the
pushforward operation
$
\mu_s=(\tau_s)_\#\mu.
$
The proposition therefore shows that the discrete CDT linearizes signal
translations exactly as in the classical continuous CDT framework.
\end{remark}

\subsection{Composition property}

The discrete CDT also linearizes monotone compositions.

\begin{proposition}[Composition property]
\label{prop:compositionProperty}
Let \(S:\R\to\R\) be strictly increasing, and define
\[
\nu=S_\#\mu.
\]
Then
\[
T_\nu(y_j)=S(T_\mu(y_j)),
\qquad j=1,\dots,m.
\]
Equivalently,
\[
T_\nu=S\circ T_\mu
\]
on the support of \(\sigma\).
\end{proposition}

\begin{proof}
Recall that
\[
\mu=\sum_{i=1}^n p_i\delta_{x_i},
\qquad
x_1<\cdots<x_n,
\]
with cumulative masses
$
P_i=\sum_{k=1}^i p_k.
$
By definition of the pushforward,
\[
\nu=S_\#\mu
=
\sum_{i=1}^n p_i\delta_{S(x_i)}.
\]
Since \(S\) is strictly increasing, the transformed support points remain
strictly ordered:
$
S(x_1)<\cdots<S(x_n).
$
Moreover, the cumulative mass levels associated with these ordered points
remain
$
P_i=\sum_{k=1}^i p_k.
$
By Proposition~\ref{prop:quantileFormula},
\[
T_\mu(y_j)=x_i,
\qquad
i=\min\{k:\ P_k\ge Q_j\}.
\]
Applying the same formula to \(\nu\), whose ordered atoms are
$
S(x_i)
$
with the same masses \(p_i\), gives
\[
T_\nu(y_j)=S(x_i),
\qquad
i=\min\{k:\ P_k\ge Q_j\}.
\]
The index \(i\) is the same in both formulas because it is determined only
by the cumulative masses \(P_k\) and \(Q_j\). Therefore,
\[
T_\nu(y_j)=S(x_i)=S(T_\mu(y_j)).
\]
Hence
\[
T_\nu=S\circ T_\mu
\]
on \(\supp\sigma\).
\end{proof}

The strict monotonicity assumption ensures that the ordering of the
transported support points is preserved.

\subsection{Scaling property}

\begin{proposition}[Scaling property]
Let \(a>0\), and define the scaled target measure
\[
\mu_a=(D_a)_\#\mu,
\qquad
D_a(x)=ax.
\]
Then the corresponding discrete CDT satisfies
\[
T_{\mu_a}(y_j)
=
a\,T_\mu(y_j),
\qquad
j=1,\dots,m.
\]
Equivalently,
\[
T_{\mu_a}
=
aT_\mu,
\]
where the multiplication is understood pointwise on the support of
\(\sigma\).
\end{proposition}

\begin{proof}
Since
\[
\mu_a=(D_a)_\#\mu,
\]
the scaling map \(D_a(x)=ax\) is strictly increasing for \(a>0\).
Therefore, Proposition~\ref{prop:compositionProperty} applies and yields
\[
T_{\mu_a}
=
D_a\circ T_\mu.
\]
Hence, for every \(j=1,\dots,m\),
\[
T_{\mu_a}(y_j)
=
D_a(T_\mu(y_j))
=
a\,T_\mu(y_j).
\]
\end{proof}

\subsection{Consistency under refinement}

While exact measure reconstruction can fail at coarse resolutions due to atomic mass mismatch (as established in Theorem~\ref{thm:exactReconstruction}), we now show that this obstruction vanishes asymptotically as the reference grid is refined.
Let \(\mu=\sum_{i=1}^n p_i\delta_{x_i}\) be a fixed target probability measure. For each \(N\in\mathbb N\), let \(\sigma_N=\sum_{j=1}^{m_N} q_j^{(N)}\delta_{y_j^{(N)}}\) denote a reference probability measure with strictly positive masses satisfying \(\sum_{j=1}^{m_N} q_j^{(N)}=1\). Let \(T_N=\mathcal C_{\sigma_N}(\mu)\) be the corresponding discrete CDT, and let \(\widetilde\mu_N=(T_N)_\#\sigma_N\) denote the reconstructed measure. We assume that the reference grid undergoes asymptotic refinement:
\[
\delta_N \coloneqq \max_{1\le j\le m_N} q_j^{(N)}\to 0 \qquad \text{as } N\to\infty.
\]

\begin{theorem}[Reconstruction consistency under refinement]
\label{thm:refinementConsistency}
Under the refinement assumption, the cumulative distribution functions of the reconstructed measures converge uniformly to the target cumulative distribution function:
\[
\lim_{N\to\infty} \|F_{\widetilde\mu_N} - F_\mu\|_\infty = 0.
\]
Consequently, the reconstructed measures converge weakly to the target measure: \(\widetilde\mu_N \rightharpoonup \mu\).
\end{theorem}
\begin{proof}
Let \(F_\mu\) and \(F_{\widetilde\mu_N}\) denote the cumulative
distribution functions of \(\mu\) and \(\widetilde\mu_N\), respectively.
Fix \(x\in\R\), and define
\[
i(x)=\max\{i:\ x_i\le x\},
\]
with the convention \(i(x)=0\) if \(x<x_1\). Then
\[
F_\mu(x)=P_{i(x)},
\qquad
P_i=\sum_{k=1}^i p_k,
\qquad
P_0=0.
\]
By the discrete quantile formula established earlier,
\[
T_N(y_j^{(N)})=F_\mu^{-1}(Q_j^{(N)}),
\qquad
Q_j^{(N)}=\sum_{\ell=1}^j q_\ell^{(N)}.
\]
Using the generalized inverse identity
\[
F_\mu^{-1}(s)\le x
\quad\Longleftrightarrow\quad
s\le F_\mu(x),
\]
we obtain
\[
T_N(y_j^{(N)})\le x
\quad\Longleftrightarrow\quad
Q_j^{(N)}\le F_\mu(x).
\]
Therefore,
\[
F_{\widetilde\mu_N}(x)
=
\sum_{\{j:\ Q_j^{(N)}\le F_\mu(x)\}} q_j^{(N)}.
\]
Let \(J_N(x)=\max\{j:\ Q_j^{(N)}\le F_\mu(x)\}\), with the convention
that the sum is zero if this set is empty. Then
\[
F_{\widetilde\mu_N}(x)=Q_{J_N(x)}^{(N)}.
\]
Since \(Q_{J_N(x)}^{(N)}\le F_\mu(x)\), we have
$
F_{\widetilde\mu_N}(x)\le F_\mu(x).
$
Moreover, by maximality of \(J_N(x)\), either \(J_N(x)=m_N\), in which
case \(Q_{m_N}^{(N)}=1\le F_\mu(x)\). Since \(F_\mu(x)\le 1\), this implies
\(F_\mu(x)=1\), and hence
$
F_{\widetilde\mu_N}(x)=F_\mu(x)=1.
$
Otherwise,
\[
Q_{J_N(x)+1}^{(N)}>F_\mu(x).
\]
In the latter case,
\[
Q_{J_N(x)+1}^{(N)}
=
Q_{J_N(x)}^{(N)}+q_{J_N(x)+1}^{(N)},
\]
and therefore
\[
F_\mu(x)-Q_{J_N(x)}^{(N)}
<
q_{J_N(x)+1}^{(N)}.
\]
Using
\[
F_{\widetilde\mu_N}(x)=Q_{J_N(x)}^{(N)},
\]
we obtain
\[
F_\mu(x)-F_{\widetilde\mu_N}(x)
<
q_{J_N(x)+1}^{(N)}
\le
\max_{1\le j\le m_N}q_j^{(N)}.
\]
Therefore, 
\[
|F_{\widetilde\mu_N}(x)-F_\mu(x)|
\le
\max_{1\le j\le m_N}q_j^{(N)}.
\]
Because this inequality holds for all $x \in \R$ uniformly, we obtain the global bound:
\[
\|F_{\widetilde\mu_N} - F_\mu\|_\infty \le \delta_N.
\]
As $\delta_N \to 0$, the uniform norm vanishes, establishing uniform convergence of the distribution functions. By standard probability theory, uniform convergence of distribution functions implies weak convergence of the underlying measures, completing the proof.
\end{proof}

The above theorem shows that although exact reconstruction may fail for coarse
reference measures due to atomic mass incompatibility, this obstruction
is purely a finite-resolution phenomenon.
As the reference measure is refined, the reconstruction error vanishes
uniformly, confirming that deterministic discrete optimal transport is
asymptotically consistent.

\section{Discrete Signed Cumulative Distribution Transform}
\label{sec:scdt}

\subsection{Signed discrete signals}

Let
\[
f=\sum_{i=1}^n a_i\delta_{x_i},
\qquad
x_1<\cdots<x_n,
\]
be a signed discrete signal, where \(a_i\in\R\).
Define the positive and negative coefficient parts by
\(
a_i^+=\max\{a_i,0\}
\)
and
\(
a_i^-=\max\{-a_i,0\}.
\)
Then
\[
f=f^+-f^-,
\qquad
\mbox{where}
\qquad
f^+
=
\sum_{i=1}^n a_i^+\delta_{x_i},
\qquad
f^-
=
\sum_{i=1}^n a_i^-\delta_{x_i}.
\]
Both \(f^+\) and \(f^-\) are finite positive atomic measures.
Their total masses are
\[
m^+=\sum_{i=1}^n a_i^+,
\qquad
m^-=\sum_{i=1}^n a_i^-.
\]
Whenever \(m^\pm>0\), the normalized probability measures are defined by
\[
\mu^\pm=\frac1{m^\pm}f^\pm.
\]

\subsection{Definition of the discrete SCDT}

Let
$
\sigma=\sum_{j=1}^m q_j\delta_{y_j}
$
be a fixed reference measure. The discrete signed cumulative distribution
transform (SCDT) is obtained by applying the discrete CDT separately to
the normalized positive and negative parts.

\begin{definition}[Discrete SCDT]
Assume \(m^+>0\) and \(m^->0\). The discrete SCDT of \(f\) relative to
\(\sigma\) is defined by
\[
\mathcal S_\sigma(f)
=
\left(
m^+,\,
T_{\mu^+},\,
m^-,\,
T_{\mu^-}
\right),
\]
where
\[
T_{\mu^+}
=
F_{\mu^+}^{-1}\circ F_\sigma,
\qquad
T_{\mu^-}
=
F_{\mu^-}^{-1}\circ F_\sigma.
\]
\end{definition}
If one of the masses \(m^\pm\) is zero, the corresponding component is
omitted or recorded as empty.

\subsection{Forward discrete SCDT algorithm}

\begin{algorithm}[H]
\caption{Discrete signed cumulative distribution transform}
\begin{algorithmic}[1]

\Require
Signed discrete signal
$
f=\sum_{i=1}^n a_i\delta_{x_i},
$
$x_1<\cdots<x_n,$ and reference measure
$
\sigma=\sum_{j=1}^m q_j\delta_{y_j}.
$

\Ensure
Discrete SCDT representation
\[
\mathcal S_\sigma(f)
=
(m^+,T^+,m^-,T^-).
\]

\State Compute \quad
$
a_i^+=\max\{a_i,0\},
\qquad
a_i^-=\max\{-a_i,0\}.
$

\State Compute \quad
$
m^+=\sum_{i=1}^n a_i^+,
\qquad
m^-=\sum_{i=1}^n a_i^-.
$

\If{\(m^+>0\)}
    \State Normalize the positive masses: \quad
    $
    p_i^+=\frac{a_i^+}{m^+}.
    $

    \State Define the normalized positive measure
    \quad
    $
    \mu^+
    =
    \sum_{i=1}^n p_i^+\delta_{x_i}.
    $

    \State Compute \quad
    $
    T^+ = \mathcal C_\sigma(\mu^+)
    $
    using the discrete CDT algorithm.
\Else
    \State Set \(T^+\) to be empty.
\EndIf

\If{\(m^->0\)}
    \State Normalize the negative masses: \quad
    $
    p_i^-=\frac{a_i^-}{m^-}.
    $

    \State Define the normalized negative measure
    \quad
    $
    \mu^-
    =
    \sum_{i=1}^n p_i^-\delta_{x_i}.
    $

    \State Compute \quad
    $
    T^- = \mathcal C_\sigma(\mu^-)
    $
    using the discrete CDT algorithm.
\Else
    \State Set \(T^-\) to be empty.
\EndIf

\State \Return \((m^+,T^+,m^-,T^-)\).

\end{algorithmic}
\end{algorithm}

\subsection{Inverse discrete SCDT}

Given a discrete SCDT representation
\(
(m^+,T^+,m^-,T^-),
\)
the canonical inverse reconstruction is obtained by reconstructing the
positive and negative channels separately and then combining them
algebraically. In particular,  
we reconstruct the normalized positive and negative pushforward measures
by
\[
\widetilde\mu^+
=
(T^+)_\#\sigma,
\qquad
\widetilde\mu^-
=
(T^-)_\#\sigma.
\]
The reconstructed signed signal is then defined by
\[
\widetilde f
=
m^+\widetilde\mu^+
-
m^-\widetilde\mu^-.
\]
Equivalently,
\[
\widetilde f
=
m^+
\sum_{j=1}^m q_j\delta_{T^+(y_j)}
-
m^-
\sum_{j=1}^m q_j\delta_{T^-(y_j)}.
\]
In general,
\[
\widetilde f\neq f,
\]
because the positive and negative parts are each subject to the same
atomic mass-compatibility issue as the unsigned discrete CDT.

\begin{algorithm}[H]
\caption{Inverse discrete signed cumulative distribution transform}
\begin{algorithmic}[1]

\Require
Reference measure
$
\sigma=\sum_{j=1}^m q_j\delta_{y_j},
$
and SCDT representation
$
(m^+,T^+,m^-,T^-).
$

\Ensure
Reconstructed signed discrete signal \(\widetilde f\).

\If{\(m^+>0\)}
    \State Compute the reconstructed normalized positive measure
    \quad
    $
    \widetilde\mu^+=(T^+)_\#\sigma
    $
    using the inverse discrete CDT algorithm.

    \State Rescale the positive part
    \quad
    $
    \widetilde f^+ = m^+\widetilde\mu^+.
    $
\Else
    \State Set \(\widetilde f^+=0\).
\EndIf

\If{\(m^->0\)}
    \State Compute the reconstructed normalized negative measure
    \quad
    $
    \widetilde\mu^-=(T^-)_\#\sigma
    $
    using the inverse discrete CDT algorithm.

    \State Rescale the negative part
    \quad
    $
    \widetilde f^- = m^-\widetilde\mu^-.
    $
\Else
    \State Set \(\widetilde f^-=0\).
\EndIf

\State Form the signed reconstruction
\quad
$
\widetilde f=\widetilde f^+-\widetilde f^-.
$

\State \Return \(\widetilde f\).

\end{algorithmic}
\end{algorithm}

\begin{remark}
\rm 
In practical implementations, atoms occurring at identical spatial
locations may be aggregated algebraically in order to obtain a simplified
canonical atomic representation. For example,
\[
0.7\delta_x-0.2\delta_x
=
0.5\delta_x.
\]
This aggregation step does not alter the reconstructed signed measure
itself.
\end{remark}

\subsection{Translation property}

The discrete SCDT inherits the translation property from the unsigned
discrete CDT.

\begin{proposition}[Translation property for the discrete SCDT]
Let
$
f
=
\sum_{i=1}^n a_i\delta_{x_i},
$
and define the translated signal
$
f_a
=
\sum_{i=1}^n a_i\delta_{x_i+a},
\qquad a\in\R.
$
Then
\[
m_a^+=m^+,
\qquad
m_a^-=m^-,
\]
and
\[
T_a^+(y_j)=T^+(y_j)+a,
\qquad
T_a^-(y_j)=T^-(y_j)+a.
\]
\end{proposition}

\begin{proof}
Translation changes the atom locations from \(x_i\) to \(x_i+a\), but it
does not change the signed coefficients \(a_i\). Hence the positive and
negative masses remain unchanged:
\[
m_a^+=m^+,
\qquad
m_a^-=m^-.
\]
Moreover, the normalized positive and negative measures are translated by
the map
\[
\tau_a(x)=x+a.
\]
Applying the translation property of the discrete CDT separately to
\(\mu^+\) and \(\mu^-\) yields
\[
T_a^+=T^++a,
\qquad
T_a^-=T^-+a.
\]
\end{proof}

\subsection{Zero-crossing and noise issues}

The discrete SCDT introduces an additional issue not present in the
unsigned CDT. Since the transform is computed separately on the positive
and negative parts, small perturbations near zero may change the sign of
an atom. Consequently, a coefficient may move from the positive channel
to the negative channel, or conversely.

In the continuous setting, such changes often involve only a very small
amount of mass. In the fully discrete setting, however, a single sign
flip may move an entire atom between the two channels. Thus the discrete
SCDT may become sensitive near zero crossings.

A simple practical remedy is to introduce a dead zone around zero.
Given a tolerance \(\varepsilon>0\), define the thresholded coefficients
by
\[
a_i^{+,\varepsilon}
=
\begin{cases}
a_i, & a_i>\varepsilon,\\
0, & |a_i|\le\varepsilon,\\
0, & a_i<-\varepsilon,
\end{cases}
\quad  
\mbox{and}
\quad 
a_i^{-,\varepsilon}
=
\begin{cases}
0, & a_i>\varepsilon,\\
0, & |a_i|\le\varepsilon,\\
-a_i, & a_i<-\varepsilon.
\end{cases}
\]
The thresholded SCDT is then obtained by replacing \(a_i^\pm\) with
\(a_i^{\pm,\varepsilon}\) before normalization. This prevents small
noise-induced sign changes from producing spurious positive or negative
atoms.

\section{Numerical examples}
\label{sec:numerics}

We now present several simple numerical examples illustrating the
behavior of the discrete CDT and discrete SCDT. All computations are
performed directly at the level of atomic measures using the algorithms
developed in the previous sections. No interpolation, density estimation,
or numerical inversion is used.

\subsection{Empirical measures and translation linearization}

One of the principal advantages of the discrete CDT framework is that it
applies directly to empirical measures.
Suppose
\[
x_1,\dots,x_n
\]
are samples drawn from an underlying probability distribution.
The associated empirical measure is
\[
\mu_n
=
\frac1n\sum_{i=1}^n \delta_{x_i}.
\]
Because the discrete CDT operates entirely through cumulative mass
matching, the transform can be computed directly from the empirical
measure itself without constructing a continuous density approximation.

In this experiment, we generate empirical samples from a Gaussian
distribution and compare the CDT of the empirical measure with the CDT of
a translated empirical measure.
More precisely, if
\[
\mu_n^{(a)}
=
(\tau_a)_\#\mu_n,
\qquad
\tau_a(x)=x+a,
\]
then the translation property predicts
\[
T_{\mu_n^{(a)}}(y_j)
=
T_{\mu_n}(y_j)+a.
\]
Figure~\ref{fig:empiricalCDT} illustrates this phenomenon.
The top panel shows the empirical distributions in physical space.
The bottom panel shows the corresponding CDT coordinates plotted against
the quantile coordinate
\[
Q_j=\sum_{\ell=1}^j q_\ell.
\]
While translation produces a nonlinear geometric shift in physical space,
the corresponding CDT representations differ essentially by a constant
vertical translation.
This demonstrates the linearization property of the discrete CDT directly
at the level of empirical atomic measures.

\begin{figure}[h!]
\centering
\includegraphics[width=.65\textwidth]{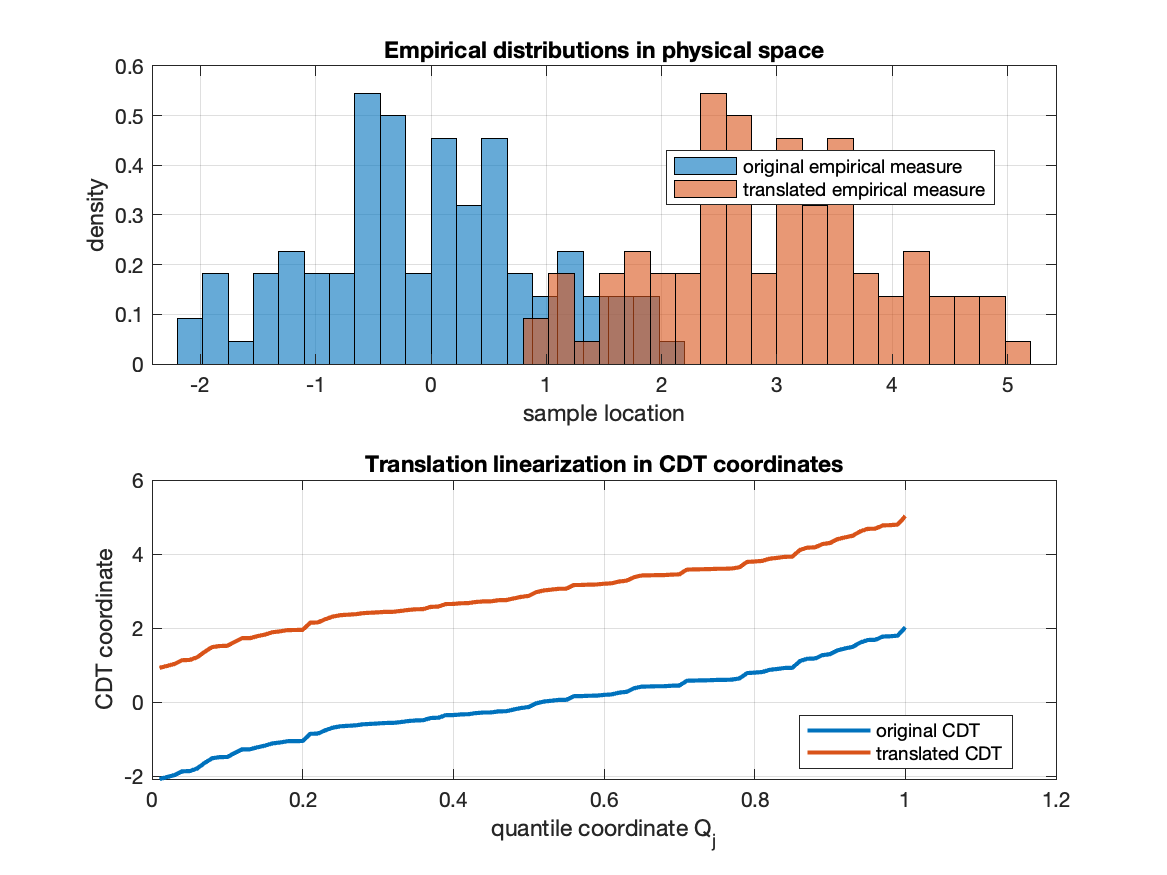}
\caption{
Empirical discrete CDT and translation linearization.
Top: empirical distributions in physical space.
Bottom: corresponding CDT coordinates.
Translation in physical space becomes an additive shift in CDT space.
}
\label{fig:empiricalCDT}
\end{figure}

\subsection{Compatibility and reference refinement}

We next investigate how reconstruction depends on the compatibility between the target cumulative mass levels and the reference cumulative levels.
Consider the target measure
\[
\mu
=
0.3\,\delta_{-1}
+
0.7\,\delta_{1}.
\]
By Theorem~\ref{thm:exactReconstruction}, exact reconstruction holds if
and only if the target cumulative mass level
\[
P_1=0.3
\]
coincides with one of the reference cumulative levels.
To examine this criterion computationally, we consider the family of
uniform reference measures
\[
\sigma_N
=
\frac1N\sum_{j=1}^N \delta_{y_j},
\]
for increasing values of \(N\).
For each \(N\), we compute the discrete CDT and reconstruct the
associated pushforward measure
\[
\widetilde\mu_N
=
(T_N)_\#\sigma_N.
\]

\begin{figure}[h!]
\centering
\includegraphics[width=.65\textwidth]{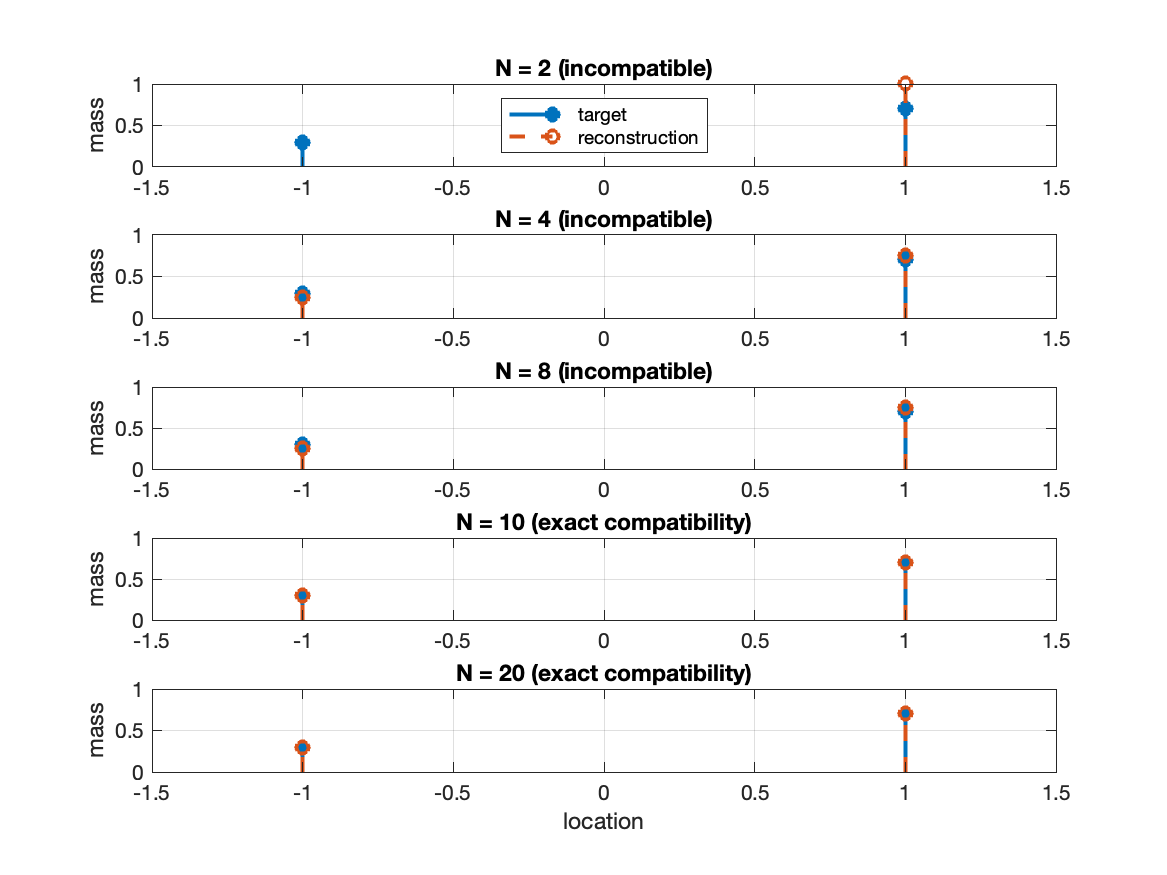}
\caption{
Compatibility and reference refinement.
Exact reconstruction occurs precisely when the target cumulative mass
level is represented among the reference cumulative levels, as predicted
by Theorem~\ref{thm:exactReconstruction}.
}
\label{fig:compatibility}
\end{figure}

Figure~\ref{fig:compatibility} illustrates the reconstruction behavior
for several reference resolutions.
For incompatible resolutions, the target mass level \(0.3\) lies strictly
between consecutive reference cumulative levels, so exact reconstruction
is impossible because deterministic transport cannot split reference
atoms.
The reconstructed measure therefore approximates the target by assigning
the nearest realizable aggregate mass.

At the compatible resolution \(N=10\), the reference cumulative level
\[
Q_3=\frac{3}{10}=0.3
\]
matches the target cumulative level exactly, and exact reconstruction is
recovered.
This exact agreement persists for all finer compatible refinements.

This experiment illustrates the algorithmic content of
Theorem~\ref{thm:exactReconstruction}: reconstruction accuracy is
governed by cumulative-mass compatibility, while reference refinement
reduces cumulative quantization error and restores exact invertibility
whenever the target cumulative levels are resolved by the reference
discretization. More generally, even when exact compatibility is absent
at finite resolution, the asymptotic convergence predicted by
Theorem~\ref{thm:refinementConsistency} ensures that the reconstructed
measures converge weakly to the target under sufficiently fine
refinement.

\subsection{Translation linearization for the discrete SCDT}

We next illustrate the discrete signed cumulative distribution  transform
and its exact translation linearization property. Figure~\ref{fig:scdt_translation} shows a signed discrete signal
consisting of positive and negative Gaussian-type components together
with a translated copy of the same signal.

The lower panel shows the corresponding positive and negative SCDT
components. Both components undergo the same additive shift in SCDT
coordinates:
\[
T_a^+=T^++a,
\qquad
T_a^-=T^-+a.
\]
Thus the discrete SCDT inherits the translation linearization property
from the unsigned discrete CDT.

\begin{figure}[h!]
\centering
\includegraphics[width=0.65\textwidth]{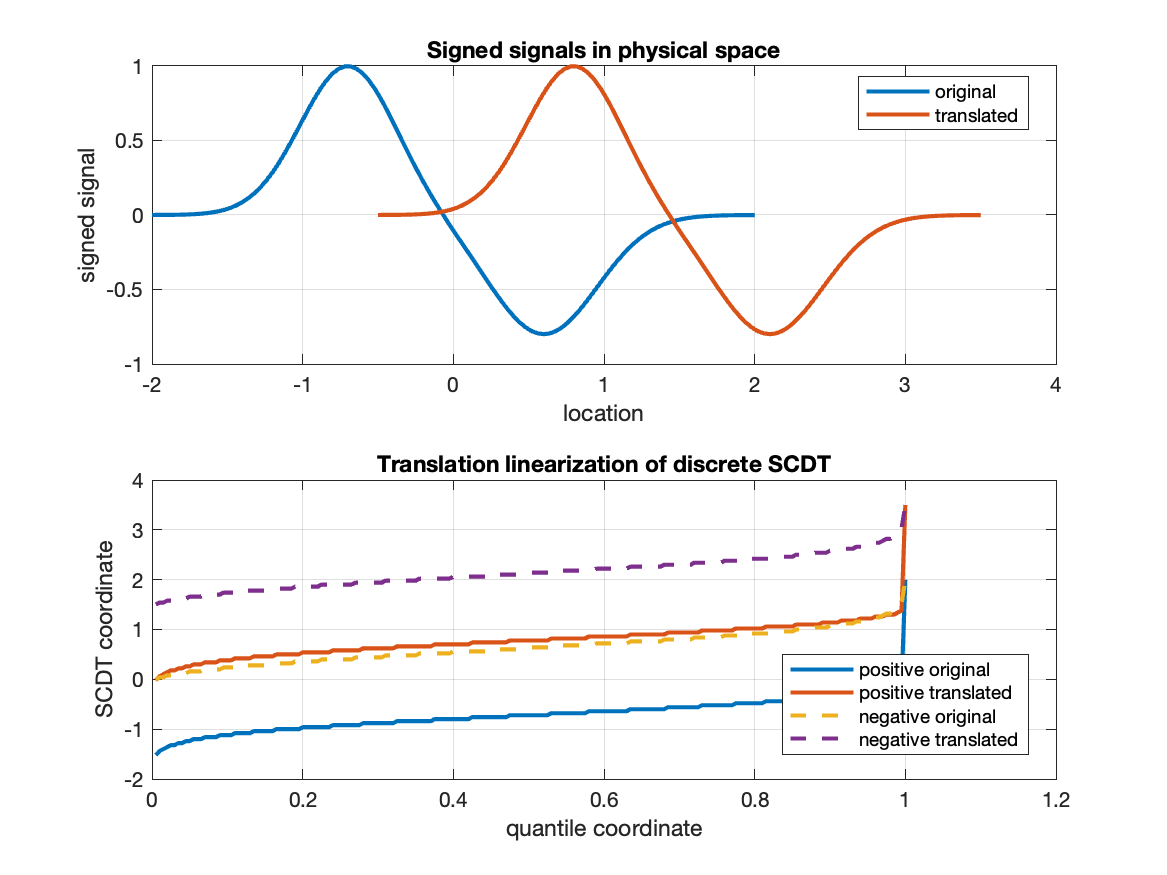}
\caption{Translation linearization for the discrete SCDT. The positive
and negative transport components both shift additively under spatial
translation of the signed signal.}
\label{fig:scdt_translation}
\end{figure}

\subsection{Zero-crossing sensitivity and thresholding}

Finally, we illustrate the sensitivity of the discrete SCDT near
zero crossings introduced in Section~\ref{sec:scdt}. We consider a signed signal with a zero crossing at the
origin and perturb it with small random noise. In the fully discrete
setting, noise near zero may change the sign of individual atoms, causing
them to move between the positive and negative transport channels.

Figure~\ref{fig:scdt_threshold} compares the reconstruction obtained
using a hard sign split with the reconstruction obtained using a
dead-zone threshold. The thresholded version ignores coefficients
satisfying
\[
|a_i|\le\varepsilon,
\]
thereby suppressing small noise-induced sign oscillations near the zero
crossing.

In the numerical experiment, the dead-zone threshold removes only a small
number of near-zero atoms while leaving the total positive and negative
masses essentially unchanged. This produces a more stable positive and
negative decomposition without significantly altering the global signal
structure.

\begin{figure}[h!]
\centering
\includegraphics[width=0.65\textwidth]{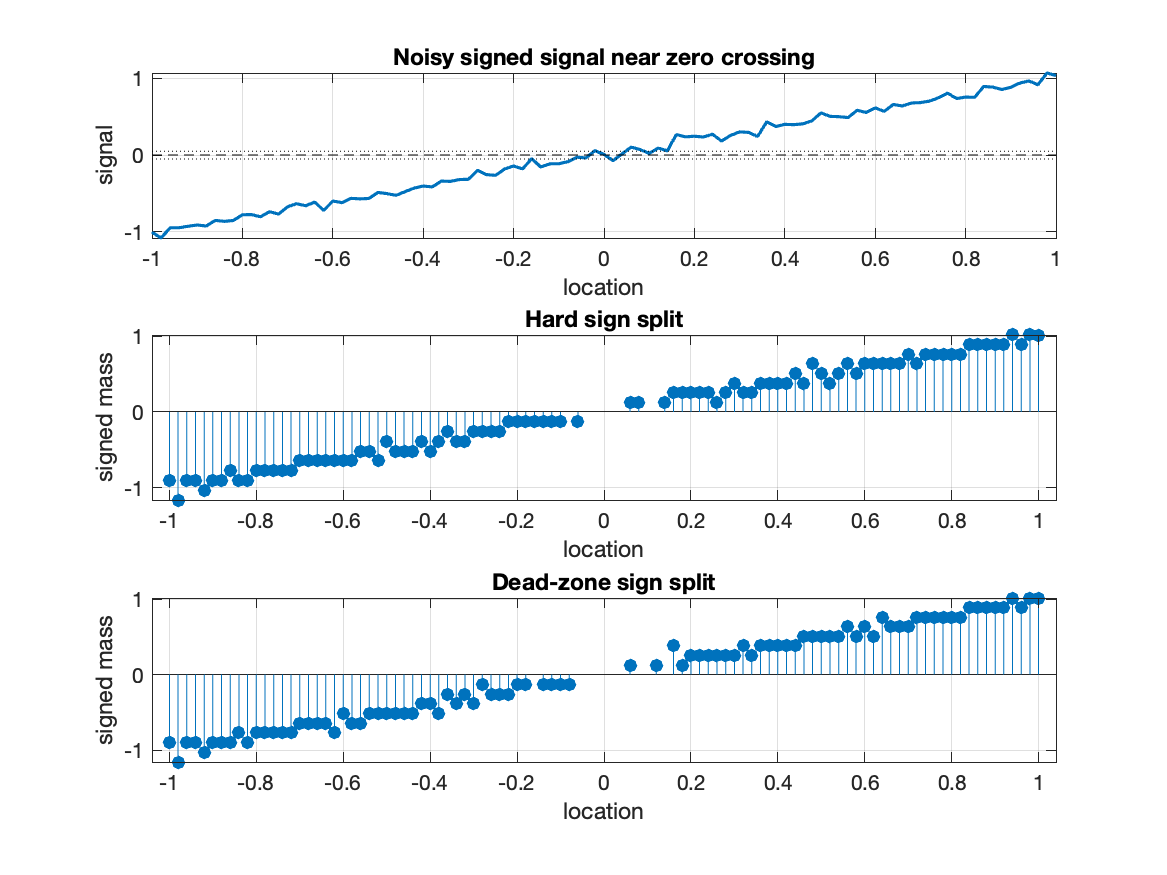}
\caption{Zero-crossing sensitivity for the discrete SCDT. The dead-zone
threshold suppresses small noise-induced sign changes near the zero
crossing, producing a more stable positive/negative decomposition.}
\label{fig:scdt_threshold}
\end{figure}

These numerical experiments confirm the principal theoretical features of
the proposed discrete transport framework: exact translation
linearization, approximate reconstruction under reference refinement, and
stabilized signed transport representations under thresholded
positive-negative decomposition.

\section{Conclusion}

This paper developed a fully discrete cumulative distribution transform
(CDT) for atomic probability measures on the real line based on monotone
quantile transport. The resulting framework admits explicit
cumulative-mass characterizations and linear-time algorithms for both
forward transformation and inverse reconstruction, requiring no
interpolation or density estimation.

A central contribution is the identification of a precise compatibility
criterion governing exact reconstruction. Unlike the continuous setting,
deterministic transport between atomic measures cannot generally split
masses, leading to an intrinsic finite-resolution obstruction. Exact
recovery was shown to occur if and only if the target cumulative mass
levels are represented among the reference cumulative levels. It was
further proved that this obstruction disappears asymptotically under
reference refinement, yielding weak convergence of reconstructed
measures to the target.

In addition, several structural properties of the discrete CDT were
established, including translation, composition, and scaling laws. The
framework was also extended to a discrete signed cumulative
distribution transform equipped with a thresholded positive-negative
decomposition that improves robustness near zero crossings.

Taken together, these results provide a rigorous transport-based
representation framework for atomic and signed discrete data. Possible
directions for future work include higher-dimensional extensions,
adaptive reference design based on cumulative-mass compatibility, and
applications to transport-based representation learning, signal filtering, 
communications, compression, inverse problems, and data-driven modeling.

\bibliographystyle{plain}
\bibliography{references}

\end{document}